\documentclass[a4paper,10pt,intlimits,ngerman]{amsart}
\usepackage[varg]{txfonts}
\usepackage[unicode]{hyperref}
\usepackage{amsfonts}
\usepackage[mathscr]{euscript}
\usepackage{graphicx}
\usepackage{color}
\usepackage{tikz-cd}
\usetikzlibrary{calc,patterns,angles,quotes}


\newtheorem{theorem}{Theorem}[section]
\newtheorem{proposition}[theorem]{Proposition}
\newtheorem{lemma}[theorem]{Lemma}

\hypersetup{
breaklinks=true,
colorlinks=true,
linkcolor=blue,
citecolor=blue,
urlcolor=blue,
}

\numberwithin{equation}{section}

\DeclareMathOperator{\hol}{h\ddot{o}l}

\title[]{Analyticity for Solution of {I}ntegro-{D}ifferential {O}perators }
\author{Simon Blatt}
\address[Simon Blatt]{Departement of Mathematics, Paris Lodron Universit\"at Salzburg, Hellbrunner Strasse 34, 5020 Salzburg, Austria}
\email[Simon Blatt]{simon.blatt@sbg.ac.at}


\date{\today}

\begin{document}

\maketitle

\begin{abstract}

We prove that for a certain class of kernels $K(y)$ that viscosity solutions of 
the integro-differential equation
$$
 \int_{\mathbb R^n} (u(x+y) - 2 u(x) + u(x-y)) K(y)  dy = f(x,u(x))
$$
are locally analytic if $f$ is an analytic function. This extends results in
\cite{Albanese2015} in which it was shown that such solutions belong to
certain Gevrey classes. 
\end{abstract}

\tableofcontents

\section{Introduction}

%
%

Non-local equations play an important role in so different fields as the modeling of 
american option prices, geometric repulsive potential, the propagation of flames, 
and particel physics, where the Boltzman equation and the Kac equation are prominent 
examples of fractional partial differential equations.

Though in  recent years the research on non-local partial differential 
equations exploded, still quite a lot of very basic questions regarding this type 
of equations remain open that have long been settled in the classical setting. In 
this article we address one of these questions: Is the solution to an elliptic 
fractional partial differential equation with analytic right-hand side analytic? 

%
%

For classical non-linear partial differential equations this is David Hilbert's 19th 
problem. Already shortly after, Bernstein  could give an answer in \cite{Bernstein1904} 
for elliptic equations in two independent variables under the  assumption that the solution is already
$C^3$  and by Petrowsky to systems \cite{Petrowsky1939}.  
Different methods of proof and generalization can be found in  \cite{Gevrey1918, Lewy1929, Hopf1932, Friedman1958, Morrey1957,Morrey1958,Morrey1958a}

In recent years some results on analyticity for special fractional equations on the whole space $\mathbb R^n$ 
or compact manifolds like $\mathcal S^1$ appeared \cite{DallAcqua2012,Barbaroux2017,Blatt2018a}.  
To the best of the authors knowledge, the findings in \cite{Albanese2015} are the only attempt to consider 
analyticity of local solutions to general fractional partial differential equations.  They prove that the solution belong 
to certain Gevrey classes but did not succeed in proving that the solutions are indeed analytic.

%
%

Let us formulate the main result of this article. We consider translation invariant kernels $K \in C^\infty( \mathbb R^n \setminus 
\{0\}, (0,\infty)) $ close to a kernel of fractional Laplacian type in the sense that
 \begin{equation} \label{eq:NearFractional}
  \left|\frac{|y|^{n+s}K(y)}{2-s} - a_0 \right| \leq \eta
 \end{equation}
 for all $y \in \mathbb R^n \setminus \{0\}.$ Here, 
$\eta> 0$ is going to be a small constant that will be determined 
later on.

 For such kernels and functions $u \in L^\infty(\mathbb R^n, \mathbb R)$ we define 
the operator
 $$
  K u(x)  = p.v. \int_{\mathbb R^n} (u(x+y) - 2 u(x) + u(x-y)) K(y)  dy.
 $$
We will furthermore assume that the kernel satisfies the estimate
\begin{equation} \label{eq:DerivativesKernel}
|\partial^\alpha_y K(y) | \leq  C \frac { 
H^{|\alpha|}
  |\alpha| !} {|y|^{n+s+|\alpha|}} \quad \text{ on } B_1(0)
\end{equation}
for all multiindices $\alpha \in \mathbb N_0^n$. We will assume without loss of generality that $H \geq 1.$
In this short note we will prove the following result.

\begin{theorem} \label{thm:AnalyticityOfSolutions}
 For $s \in (1,2)$ let us assume that $u \in L^\infty(\mathbb R^n, \mathbb R) \cap C^\infty(B_1(0))$ is a 
 viscosity solution of the equation
 $$
  K u (x) = f(x,u(x))
 $$
 for an analytic function $f:B_1(0) \times \mathbb R \rightarrow \mathbb R$. Then $u$ is analytic on $B_1(0).$
\end{theorem}

Note that in view of the bootstrapping argument in \cite{Barrios2014} the assumption $u \in C^\infty(B_1(0))$ is not essential. In contrast to \cite{Albanese2015} we only consider translation invariant 
equations here. But this is not the reason why the result stated here 
is stronger: Unfortunately some of the additional terms coming from $x$-dependence 
of the kernel seem to be missing in \cite[inequality (3.2)]{Albanese2015} and hence their 
proof seems to at least have a gap. Though we believe that also these additional 
terms can be controlled we leave this case for a later paper as this will be technically 
more involved.

%
%

As in \cite{Albanese2015}, we proof Theorem 
\ref{thm:AnalyticityOfSolutions} combining the classical approaches by Friedman and
Morrey with the a-priori estimates for solution in \cite{Caffarelli2011}. In 
contrast to \cite{Albanese2015} we omit the use of 
incremental differences and discrete partial integration completely and directly 
work with partial derivatives and partial integration. The essential new ingredient in our proof is to estimate the terms coming 
from the long-range interactions of the equation in much more sophisticated way using nested balls.
 
In Section \ref{sec:Preliminaries} we gather some known facts and tools for the proof of Theorem \ref{thm:AnalyticityOfSolutions}, i.e. a characterization of analyticity, the Schauder estimates of Caffarelli and Silvestre in \cite{Caffarelli2011} and an elementary estimate for the binomial. The essential estimate for higher derivative is then derived in Section \ref{sec:APriori} before we turn to the proof of Theorem \ref{thm:AnalyticityOfSolutions} in Sections \ref{sec:Proof1} and \ref{sec:Proof2}. In \ref{sec:Proof1} we give the proof first for the special case that the right-hand side of our equation does not depend on $x$ and not on $u$. We do this for two reasons: To make the presentation  as readable as possible and since this special case contains the major new difficulties. We will then see in Section \ref{sec:Proof2} that one can deal with the $u$-dependence by applying a higher order chain rule in a fairly standard way.

%
%

\section{Preliminaries} \label{sec:Preliminaries}

\subsection{Characterization of Analytic Functions}


The following fact is well known.

\begin{theorem} A function $u:\Omega \rightarrow \mathbb R$ is analytic on $\Omega$, $\Omega \subset \mathbb R^n$ open, 
if and only if for every compact set $K \subset \Omega$ there are constants $C=C_k,A=A_K < 
\infty$ such that 
 $$
  \|\nabla^k u \|_{L^\infty(B_r(x))} \leq C A^k k!
 $$
 for all $k \in \mathbb N_0$.
\end{theorem}

A proof of this theorem can be found in \cite{Krantz1992}.

\subsection{A-Priori Estimates for Non-Local Integro-Differential Operators}


 Caffarelli and Silvestre proved the following remarkable theorem.

\begin{theorem}[\protect{\cite[Theorem 61]{Caffarelli2011}}] 
\label{thm:SchauderEstimate}
Let $s \in (1,2)$ and $u \in L^\infty(\mathbb R^n)$ be a viscosity solution of 
  $$
   Ku (x)  = f(x) \text { on } B_1(0)
  $$
  for an $f \in L^\infty(B_1(0))$ and let $\eta>0$ in \eqref{eq:NearFractional} be small enough. 
  Then for all $0 < \alpha < 1-s$ we have $u \in C^{1,\alpha}(B_{\frac 1 2}(0))$
  and
  $$ 
   \|u\|_{C^{1,\alpha} (B_{\frac12}(0))}
   \leq C\left(\|f\|_{L^\infty(B_1(0))} 
   +  \|u\|_{L^\infty(\mathbb R^n)}\right).
  $$ 
  \end{theorem}
 
 
 Scaling this result, we immediately get the following.
 
 \begin{theorem} \label{thm:ScaledSchauderEstimate}
  Let $s \in (1,2)$ and $u \in L^\infty(\mathbb R^n)$ solve 
  $$
   Ku (x) = f(x) \text{ in } B_{r}(0)
  $$
  for an $f \in L^\infty(B_r(0))$ and let $\eta>0$ in \eqref{eq:NearFractional} be small enough.
  Then for all $0 < \alpha < 1-s$ we have $u \in 
C^{1,\alpha}(B_{\frac r 2 }(0))$
  and
  $$
   r\|\nabla u\|_{L^\infty (B_{\frac r2}(0))} + 
r^{1+\alpha} \hol_{\alpha, B_{\frac r2}(0)} (\nabla  u)
   \leq C\left(r^{s}\|f\|_{L^\infty(B_1(0))} 
   +  \|u\|_{L^\infty(\mathbb R^n)}\right).
  $$
 \end{theorem}

 \subsection{An Estimate for the Binomial}  We will need the following estimate for the binomial.
 
 \begin{lemma} \label{lem:Binomial}
  We have
  $$
   \frac{k^k}{(k-l)^{k-l} l^l } \leq (2e)^l \binom kl
  $$
  for all $k \in \mathbb N$, $k>l>0$.
 \end{lemma}
 
 \begin{proof}
  For $ 0< l \leq \frac k2$ we have  
$$
 \binom{k}{l} \geq 2^{-l} \frac{k^l}{l^l}
$$
and 
$$
 \frac{k^k}{(k-l)^{k-l} l^l} =\left( \frac k {k-l} \right) ^{k-l}\frac{k^l}{l^l} = 
\left( 
1 +  \frac l {k-l} \right) ^{k-l}\frac{k^l}{l^l} \leq e^l \frac{ k^l} {l^l}.
$$
Hence,
$$
 \frac{k^k}{(k-l)^{k-l}l^l} \leq (2e)^l \binom k l 
$$
if $l \leq \frac k 2.$ For $l> \frac k2 $ we get applying the above to $k-l$ instead 
of $k$
$$
 \frac{k^k}{(k-l)^{k-1} l^l} \leq (2e)^{k-l} \binom k {k-l} \leq (2e)^l \binom k 
{l}. 
$$
 \end{proof}

 
 \section{The Essential A-Piori Estimate}  \label{sec:APriori}
  
 We use the estimates of Caffarelli and Silvestre to derive the following 
recursive estimate for derivatives of higher order. To shorten notation we 
use the shortcuts $B_R = B_R(0)$ and 
$
\|u\|_{A} = \|u\|_{L^\infty (A)}
$
for a subset $A \subset \mathbb R^n. $ Furthermore, we will use
$$
 \|\nabla ^k u\|_{A}:= \sup_{|\alpha| =k} \|\partial^{\alpha} u\|_{A}.
$$

 \begin{theorem} \label{thm:ScaledHigherOrderEstimate}
  Let $u \in L^\infty(\mathbb R^n) \cap C^\infty(\Omega)$ and $f:\Omega 
\rightarrow \mathbb R$ be smooth such that 
  $$
    Ku = f \text{ on } \Omega.
  $$ If $x_0 \in \Omega$, $\sigma>0$. and $k\in \mathbb N$ are chosen such that 
$B_{6\sigma(k+1)}(x_0) \subset 
 \Omega$, then 
  \begin{multline*}
  \sigma \| \nabla^{k+1} u \|_{B_{\sigma}(x_0)} \leq C \bigg( \sigma^s 
  \|\nabla ^k f\|_{B_{2\sigma}(x_0)} + \|\nabla^k u\|_{B_{4r}(x_0)} \\ + \sigma^s 
\sum_{l=1}^{k-1} \frac {H^l l! 
\|\nabla^{k-l} u 
\|_{B_{6l\sigma+2\sigma} (x_0)}}{(6l\sigma)^{l+s}} + \sigma^s \frac {H^k 
k! \| u 
\|_{\mathbb R^n}} {(6k\sigma)^{k+s}} \bigg) 
  \end{multline*}
 \end{theorem}

 \begin{proof}
 After a suitable translation we can assumen that $x_0=0$. 
 We first show the statement of the theorem under the addition assumption that $u$ is 
$C^\infty$ on the complete space $\mathbb R^n$ and has compact support.
For that we chose $\tilde \eta \in C^{\infty}(\mathbb R^n,[0,1])$ such that 
 $$
  \tilde \eta \equiv 1 \text{ on } B_3  \quad \text{and} \quad \tilde 
\eta \equiv 0 \text{ on } \mathbb R^n \setminus B_4 
 $$
 and set
 $$
  \eta(x) = \tilde \eta (\tfrac x \sigma).
 $$
 For $k \in \mathbb N$ and $i_1,i_k \in \{0, \ldots, n\}$ we decompose
 $$
  w= \partial_{i_k, \ldots , i_1} u = \partial_{i_k} (\eta \partial_{i_{k-1} 
\ldots, i_1} u)  + \partial_{i_k} ((1 - \eta) \partial_{i_{k-1} 
\ldots, i_1} u)  = w_1 + w_2.
 $$
 Applying Theorem \ref{thm:ScaledSchauderEstimate} we get
 \begin{equation}\label{eq:APriori}
  \sigma \|\nabla \partial_{i_1, \ldots , i_k} u\|_{B_\sigma(0)} \leq 
  C \left( \sigma^s\|Kw_1\|_{B_{2\sigma}}  + \|w_1\|_{\mathbb R^n} 
\right) 
 \end{equation}
 We first note that
 \begin{equation}\label{eq:W1}
 \begin{aligned}
  \|w_1\|_{\mathbb R^n} &= \|\partial_{i_k} (\eta \partial_{i_{k-1} 
\ldots, i_1} u) \|_{\mathbb R^n} \leq \|\nabla^k u\|_{B_{4\sigma}} + 
\|\nabla \eta \|_{\mathbb R^n} \| 
\nabla^{k-1} u\|_{B_{4\sigma}} 
\\
&\leq \|\nabla^k u\|_{B_{4\sigma}} + \frac C \sigma  \| 
\nabla^{k-1} u\|_{B_{4\sigma}}.
 \end{aligned}
 \end{equation}
 To estimate the first term in \eqref{eq:APriori}, we use $w_1 = w - w_2$ to get
 \begin{equation} \label{}
  \|Kw_1\|_{B_{2\sigma}} \leq \|Kw\|_{ B_{2\sigma}} + 
\|Kw_2\|_{B_ {2\sigma} } \leq \|\nabla^k f\|_{B_{2r}} + \|Kw_2\|_{B_ {2\sigma} } .
 \end{equation}
and observe that for $x \in B_{2\sigma}$ we have
 \begin{equation} \label{eq:Kw2}
 \begin{aligned}
  |K w_2 &(x)| =  \left|\int_{\mathbb R^n} (w_2(x+y) - 2 w_2(x) +  w(x-y )) K (y) dy \right| \\
  & = \left|\int_{\mathbb R^n} (w_2(x+y) +  w(x-y )) K (y) dy \right| \\
  & \leq 2 \left|\int_{\mathbb R^n} w_2(x+y) K (y) dy \right| \\
  &= 2  \left| \int_{\mathbb R^n} (1-\eta(x+y)) \partial_{i_{k-1}, \ldots, i_1} u (x+y) \partial_{i_k} K (y) dy  \right|\\
  & \leq \left| \int_{B_{6\sigma}} (1-\eta(x+y)) \partial_{i_{k-1}, \ldots, i_1} u (x+y) \partial_{i_k} K (y) dy  \right|   
  \\
  & \quad \quad + \left| \int_{\mathbb R^n \setminus B_{6\sigma}} (1-\eta(x+y)) \partial_{i_{k-1}, \ldots, i_1} u (x+y) \partial_{i_k} K (y) dy  \right| 
  \\
  &= I_1 + J_1.
  \end{aligned}
  \end{equation}
 To estimate $I_1$, we note that due to the properties of $\eta$ and the triangle 
inequality $1-\eta (x+y) = 0$ if $|y| \leq \sigma$ and hence we get from the properties  
of $K$ that
\begin{equation} \label{eq:I1}
  I_1 \leq  C H \|\nabla^{k-1} u \|_{L^\infty(B_{8\sigma})} \int_{\mathbb R^n \setminus 
B_{\sigma}}  \frac 1 {|y|^{n+1+s}} dy = C \frac 
{H}{\sigma^{1+s}} \|\nabla^{k-1} u \|_{L^\infty(B_{8\sigma})} . 
\end{equation}
 For $J_1$ we use partial integration to get
 \begin{align*}
  J_1 &\leq \left| \int_{\mathbb R^n \setminus B_{6\sigma}(0)} \partial_{i_{k-2}, \ldots, 
i_1} u (x+y) \partial_{i_k, i_{k-1}} K (y) dy \right| + 
\left| \int_{\partial B_{6\sigma}(0)} |\partial_{i_{k-2}, \ldots, 
i_1} u (x+y) | |\partial_{i_k} K (y) | dS(y) \right| \\ 
 \\ &\leq C H^2  \|\nabla^{k-2} u \|_{L^\infty (B_{14\sigma})} \int_{B_{12\sigma} \setminus 
B_{6\sigma}} \frac 1 {|y|^{n+2+s}} +  C H \|\nabla^{k-2}\|_{L^\infty(B_{14\sigma})} 
\int_{\partial B_{6\sigma}} \frac 1 {|y|^{n+s+1}} dS(y) 
 \\ & +
\left| \int_{\mathbb R^n \setminus B_{12\sigma}(0)} \partial_{i_{k-2}, 
\ldots, 
i_1} u (x+y) \partial_{i_k,i_{k-1}} K (y) dy \right| 
\\
& \leq\frac { C H^2 \|\nabla^{k-2} u \|_{L^\infty (B_{14\sigma})} }{(6\sigma)^{2+s}} + J_2.
 \end{align*}
 where 
 $$
 J_2 = \left| \int_{\mathbb R^n \setminus B_{12\sigma}(0)} \partial_{i_{k-2}, 
\ldots, 
i_1} u (x+y) \partial_{i_k, i_{k-1}} K (y) dy \right|.
$$
Setting 
$$
 J_l =  \left| \int_{\mathbb R^n \setminus B_{6l\sigma}(0)} \partial_{i_{k-l}, 
\ldots, 
i_1} u (x+y) \partial_{i_k, \ldots, i_{k+1-l}} K (y) dy \right|
$$
we obtain as above using integration by parts and \eqref{eq:DerivativesKernel}
 \begin{align*}
  J_l &\leq \left| \int_{\mathbb R^n \setminus B_{6l\sigma}(0)} \partial_{i_{k-l-1}, 
\ldots, 
i_1} u (x+y) \partial_{i_k, \ldots, i_{k-l}} K (y) dy \right| \\ & \quad \quad \quad + 
\left| \int_{\partial B_{6l\sigma}(0)} |\partial_{i_{k-l-1}, \ldots, 
i_1} u (x+y) | |\partial_{i_k, \ldots, i_{k+1-l}} K (y) | dS(y) \right| \\ 
\\ 
 &\leq C H^{l+1} (l+1)! \|\nabla^{k-l-1} u \|_{B_{6(l+1)\sigma+2\sigma}} 
\int_{B_{6(l+1)\sigma} \setminus 
B_{6l\sigma}} \frac 1 {|y|^{n+l+1+s}}  \\ & \quad \quad \quad +  C H^l l! 
\|\nabla^{k-1-l}\|_{L^\infty(B_{6(l+1)\sigma+ 2\sigma})} 
\int_{\partial B_{l\sigma}} \frac 1 {|y|^{n+s+l}} dS(y) + J_{l+1}
\\ 
& \leq C H^{l+1} (l+1)! \|\nabla ^{k-(l+1)}\|_{B_{(6(l+1)\sigma)} } \frac 1 
{(6l\sigma)^{l+1+s}} + J_{l+1}.
 \end{align*}
 Iterating this estimate yields
 \begin{equation} \label{eq:J1}
 \begin{aligned}
 J_1 \leq C \sum_{l=2}^{k-1} \frac {H^l l! \|\nabla^{k-l} u 
\|_{L^\infty(B_{6l\sigma+2\sigma})}}{(6l\sigma)^{l+s}} + J_{k} \\
\leq C \left( \sum_{l=2}^{k-1} \frac {H^l l! \|\nabla^{k-l} u 
\|_{B_{6l\sigma+2\sigma}}}{(6l\sigma)^{l+s}} + \frac {H^k k! \| u 
\|_{\mathbb R^n}} {(6k\sigma)^{k+s}} \right).
 \end{aligned}
 \end{equation}
 Together the estimates \eqref{eq:APriori} -- \eqref{eq:J1} prove the statement of 
the theorem for all $u \in C^\infty (\mathbb R^n, \mathbb R)$ with compact support.

To get the statement for $u \in L^\infty (\mathbb R^n, \mathbb R) \cap 
C^\infty(\Omega, \mathbb R)$, we let $u_m $ be such that
$$
 u_m = u \text{ on } B_{m} ,
$$
and $$\|u_m\|_{L^\infty} \leq 
\|u\|_{L^\infty}.$$
We can then apply what we have proven so far to the function $u_m$ instead of $u$ to get
\begin{multline} \label{eq:um}
  \sigma \| \nabla^{k+1} u_m \|_{B_{\sigma}(x_0)} \leq C \bigg( \sigma^s 
  \|\nabla ^k f_m\|_{B_{2\sigma}(x_0)} + \|\nabla^k u_m\|_{B_{4r}(x_0)} \\ + \sigma^s 
\sum_{l=1}^{k-1} \frac {H^l l! 
\|\nabla^{k-l} u_m 
\|_{B_{6l\sigma+2\sigma} (x_0)}}{(6l\sigma)^{l+s}} + \sigma^s \frac {H^k 
k! \| u_m
\|_{\mathbb R^n}} {(6k\sigma)^{k+s}} \bigg) 
  \end{multline}
  where $f_m = Ku_m$. It is obvious that due to the properties of the approximations $u_m$ we can go to 
 the limit in the inequality and thus obtain the inequality for $u$ once we have shown that 
 $$
 \|\nabla ^k f_m\|_{B_{2\sigma}(x_0)}  \rightarrow  \|\nabla ^k f\|_{B_{2\sigma}(x_0)}  
 $$
 for $m \rightarrow \infty$.
For $x \in B_{2\sigma}$ and $\alpha \in \mathbb N^n$ with $|\alpha|=k$ we calculate
\begin{align*}
 \partial^\alpha (K u_m)(x) = \partial^\alpha Ku + \partial^\alpha K(v_m)
\end{align*}
where $v_m = u_m - u$ and using that $v_m=0$ on $B_m$  
\begin{multline*}
 \partial^\alpha K(v_m)(x) = \partial^\alpha \int_{\mathbb R^n - B_{\frac m2}} 
v_m(y) \left( K(y+x) + K(y-x)\right) dy \\ = \int_{\mathbb R^n - B_{\frac m2}} 
v_m(y) \left( \partial^{\alpha} K(y+x) + \partial^\alpha K(y-x)\right) dy 
\end{multline*}
Hence,
$$
\|\nabla^k K(v_m)\|_{B_{2\sigma}} \leq C H^{|\alpha|} |\alpha|! m^{-s-|\alpha|} \|v_m\|_{\mathbb R^n}
\leq C H^{|\alpha|} |\alpha|! m^{-s-|\alpha|} \|u\|_{\mathbb R^n} \xrightarrow{m\rightarrow \infty} \infty
$$
and thus
$$
 \|\nabla^k K u_m\|_{B_{2\sigma}} \rightarrow \|\nabla^k K u\|_{B_{2\sigma}} =  \|\nabla^k f\|_{B_{2\sigma}} 
$$
 \end{proof}

\section{Proof of the Theorem for $Ku(x) = f(x)$} \label{sec:Proof1}

Let us first illustrate this method for the special case that $Ku(x) =f(x)$, i.e. that the righthand side of our equation does not depend on $u$. 
 
 \subsection{A Recursive Estimate}

Following \cite{Albanese2015} we define the quantities
\begin{gather*}
 N_k = \sup_{0 < r < 1} \left( |1-r|^{k+s} \|\nabla^k f\|_{L^{\infty} (B_r)} \right) 
\quad \text{ for } k \geq 0 \\
 M_k = \sup_{0 < r < 1} \left( |1-r|^{k} \|\nabla^k u\|_{L^{\infty} (B_r)} \right) 
\quad \text{ for } k \geq 1, \\
 M_0 = \|\nabla^k u\|_{L^{\infty} (\mathbb R^n)}.
\end{gather*}

We will deduce the following estimate for these quantities from Theorem 
\ref{thm:ScaledHigherOrderEstimate}.

\begin{theorem} \label{thm:RecursiveEstimate}
 We have
  $$
  M_{k+1} \leq C \left(N_k + k \sum_{l=0}^{k} \binom{k}{l} M_{k-l} (2e)^l H^l l! 
\right)
  $$
for all $k \in \mathbb N_0$ and a constant $A$.
\end{theorem}

\begin{proof}
For $x_0 \in B_1(0)$ and $k \in \mathbb N$ we apply Theorem 
\ref{thm:ScaledHigherOrderEstimate} with $\sigma = \frac {1-|x_0|} {6(k+2)}$ to get
\begin{multline*}
    \| \nabla^{k+1} u \|_{B_{\sigma}} \leq C \bigg( \sigma^{s-1} 
  \|\nabla ^k f\|_{B_{2\sigma}} + \sigma^{-1} \|\nabla^k u\|_{B_{4r}} \\ + 
\sigma^{s-1} \sum_{l=2}^{k-1} \frac {H^l l! 
\|\nabla^{k-l} u 
\|_{B_{6(\sigma+2\sigma)}}}{(6l\sigma)^{l+s}} + \sigma^{s-1} \frac {H^k k! 
\| u 
\|_{\mathbb R^n}} {(6k\sigma)^{k+s}} \bigg)
\end{multline*}
where we use $B_r=B_r(x_0)$ to shorten notation. Hence,
\begin{multline*}
 (1-|x_0|)^{k+1}|\nabla^{k+1} u(x_0)|  \leq 
 (1-|x_0|)^{k+1} C \bigg( \sigma^{s-1} 
  \|\nabla ^k f\|_{B_{2\sigma}} + \sigma^{-1} \|\nabla^k u\|_{B_{4r}} \\ + 
\sigma^{s-1} \sum_{l=2}^{k-1} \frac {H^l l! 
\|\nabla^{k-l} u 
\|_{B_{6l\sigma+2\sigma}}}{(6l\sigma)^{l+s}} + \sigma^{s-1} \frac {H^k k! 
\| u 
\|_{\mathbb R^n}} {(6k\sigma)^{k+s}} \bigg .
\end{multline*}
We estimate
\begin{align*}
(1-|x_0|)^{k+1} \sigma^{s-1} \|\nabla^k f\|_{B_{2\sigma} (x_0} = (6(k+2) 
\sigma)^{k+1} \sigma ^{s-1}\|\nabla^k f\|_{B_{2\sigma} (x_0)} \\
\leq  \frac{(6(k+2) \sigma)^{k+1} \sigma ^{s-1}} {((6(k+2)-2) \sigma)^{s+k}} N_k 
 = \frac{(k+2)^{k+1}} {(k+2-\frac 1 3)^{k+s}} 
 \leq \left( 1 + \frac {\frac 1 3} 
{k+2 - \frac 1 3} \right)^{k+1} N_k
 \leq C N_k
\end{align*}
as $s > 1$ and 
$$
 0 < \left( 1 + \frac {\frac 1 3} 
{k+2 - \frac 1 3} \right)^{k+1} < \left(1 + \frac {\frac 13} {k+\frac 5 3)} 
\right)^{k+\frac 53} \rightarrow e^{\frac 1 3}.
$$
Similarly,
\begin{align*}
(1-|x_0|)^{k+1} \sigma^{-1} \|\nabla^k u\|_{B_{4\sigma} (x_0}
\leq \frac {(6(k+2)^{k+1})}{ (6(k+2)-4)^{k}} M_k
\leq C (k+2) M_k.
\end{align*}
Futhermore, we get for $1\leq l < k$
\begin{align*}
 (1-|x_0|)^{k+1} & \sigma^{s-1} \frac {H^l l! 
\|\nabla^{k-l} u 
\|_{B_{6l\sigma+2\sigma}}}{(6l\sigma)^{l+s}} 
\leq  \frac{(6(k+2))^{k+1}}{(6(k+2-l) - 2)^{k-l} (6l)^{l+s}} H^l l! M_{k-l} \\
& =6 \frac{(k+2)^{k+1}}{((k-l) + \frac{10}{6})^{k-l} l^{l+s}} H^l l! M_{k-l}
\end{align*}
Note that
\begin{align*}
 \frac{(k+2)^{k+1}} {(k-l+{\frac {10} 6})^{k-l} l^{l+1}}
 &= \left(\frac{k+2}{k+1}\right)^{k+1} \left(\frac{k-l}{k-l + \frac {10} 6} \right) 
^{k-l}  \left( \frac l {l+1}\right)^{l+1} \frac{(k+1)^{k+1}} {(k-l)^{k-l} 
(l+1)^{l+1}}
\\
&= \left(1 +\frac{1}{k+1}\right)^{k+1} \left(1 - \frac{\frac{10} 6}{k-l + \frac 
{10} 6} \right) 
^{k-l} \left( 1 - \frac{1} {l+1}\right) ^{l+1}\frac{(k+1)^{k+1}} {(k-l)^{k-l} 
(l+1)^{l+1}}
\\
& \leq
C \frac{(k+1)^{k+1}} {(k-l)^{k-1} (l+1)^{l+1}}
\leq  C  (2e)^l\binom{k+1}{l}. 
\end{align*}
In the last step we used Lemma \ref{lem:Binomial}.
Finally,
$$
 \frac{(1-|x_0|)^{k+1}}{ \sigma (6k \sigma)^k} = \frac{(6(k+2))^{k+1}}{(6k)^k} = 6 
(k+2) (1+\frac 2 k)^k \leq C 6(k+2). 
$$
These estimates show that 
$$
 (1-|x_0|)^{k+1}|\nabla^{k+1} u(x_0)| \leq C \left(N_k + k \sum_{l=0}^{k} \binom{k}{l} M_{k-l} (2e)^l H^l l! \right)
$$
Taking the supreme over all $x_0 \in B_1(0)$ proves the theorem.
\end{proof}

\subsection{The Conclusion using Cauchy's Method of Majorants}

%
%

We will now conclude the proof of Theorem \ref{thm:AnalyticityOfSolutions} using 
Cauchy's method of Majorants. 

As being analytic is a local statement, we can assume w.l.o.g that there are constants $C_f, A_f < \infty$ such that
$$
 N_k \leq C_f A_f^k k!
$$
for all $k \in \mathbb N_0$.  Setting $A := \sup\{A_k, 2eH\}$ Theorem \ref{thm:RecursiveEstimate} tells us that
\begin{equation} \label{eq:RecursiveInequality}
 M_{k+1} \leq C (N_k + k \sum_{l=0}^k \binom {k} {l} M_{k-l}(2eH)^l l!)
 \leq C  A^k k! +  C k \sum_{l=0}^k \binom {k} {l} M_{k-l}A^l l!.
\end{equation}
for all $k \in \mathbb N_0.$
We will show that this recursive estimate implies that $M_k \leq C_u A_u^k k!$ for 
suitably chosen constants $C_u, A_u$ by comparing it to the solution of an 
analytic ordinary differential equation.

For this we put
$$
 G(t) := C  \sum_{k \in \mathbb N_0} A^k t^k 
$$
and consider the solution to the initial value problem
$$
 \begin{cases}
  c'(t) = G(t) + (tG(t)c(t))'  \\
  c(0) = M_0.
 \end{cases}
$$
As near to $t$ we have $1-tG(t) \not= 0$ we can rewrite this equation as
$$
 c'(t) = \frac{2 G(t) + t G'(t)}{1-tG(t)} 
$$
near $0$. Hence, the above initial value problem has a unique analytic solution on some small time interval
$(-\varepsilon, \varepsilon).$ The derivatives $\tilde M_k = c^{(k)}(0)$ satisfy 
$$
 \tilde M_k \leq C_u A_u^k k!
$$
for suitable constants $C_u, A_u$ and 
the recursive relation
$$
 \tilde M_{k+1} = C (N_k + k \sum_{l=0}^k \binom k l  \tilde M_{k-l} (eH)^l l !)
$$ 
Comparing this with \eqref{eq:RecursiveInequality} we deduce by induction that
$$
M_{k} \leq \tilde M_{k} \leq C_u A_u^k k!.
$$

\section{Proof of the Theorem for $Ku(x) = f(x,u(x))$} \label{sec:Proof2}

Let us now move to the case that
$$
 K(u) = f(x,u(x)) \text{ in } B_1(0).
$$
As in the last section we have
  $$
  M_{k+1} \leq C \left(  N_k + k \sum_{l=0}^{k} \binom{k}{l} M_{k-l} (2e)^l H^l l! 
\right)
  $$
for all $k \in \mathbb N_0$ and a constant $A$ where now
$$
   N_k = \|\nabla^k (f(x,u(x)))\|_{B_1}.
$$
We introduce the terms
$$
 \tilde M_k = M_k +1 
$$
and 
$$
 \tilde N_k = \|\nabla^k  f\|_{K}
$$
where $K$ is the image of $x \rightarrow (x,u(x))$. As being analytic is a local property, we can again assume without loss of generality that
$$
  \tilde N_k \leq C A_f^k k! 
$$
for a constant $A_f < \infty.$ We still have
\begin{equation} \label{eq:RecursiveInequality2}
  \tilde M_{k+1} \leq C \left(  N_k + k \sum_{l=0}^{k} \binom{k}{l}  \tilde M_{k-l} (2e)^l H^l l! 
\right)
  \end{equation}

We need a higher order chain rule to estimate $N_k$ in terms of $\tilde N_k$ and tilde $M_k$.

\subsection{Higher Order Chain Rule}

\begin{proposition} \label{prop:FraaDiBruno}
 Let $g:\mathbb R^{m_1} \rightarrow \mathbb R^{m_2}$ and $f: \mathbb R^{m_2} \rightarrow \mathbb R$ be two $C^k$-functions. Then for an multiindex 
 $\alpha \in \mathbb N ^{m_1}$ of length $|\alpha| \leq k$ and $x \in \Omega$ the derivative
 $$
  \partial^\alpha (f \circ g ) (x) = P_{m_1, m_2}^\alpha ( \{\partial^\gamma f(g(x))\}_{|\gamma| \leq |\alpha|}, \{\partial^\gamma g_i\}_{0 \leq \gamma \leq \alpha} )
 $$
 where $P^\alpha_{m_1,m_2}$ is a linear combination with positive coefficients of terms of the form 
 $$
  \partial^k_{x_{i_1}, x_{i_k}}( g(x))  \partial^{\gamma_1}g_{i_1} \cdots  \partial^{\gamma_k}g_{i_k}  
 $$
 with $1 \leq k \leq |\alpha|$ and $|\gamma_1| + \ldots |\gamma_k| = |\alpha|.$.
 \end{proposition}

For $m_1=m_2=1$ we will use the notation $P^k$ instead of $P^\alpha_{m_1,m_2}.$ 
We leave the easy inductive proof of this statement to the reader. Although very precise formulas of the higher order chain rule wer give by Faa di Bruno \cite{DiBruno1857} for the univariate  case and by for example Constanini and Savits in \cite{Constantine1996} for the multivariate case, the above proposition contains all that is needed in our proof.
 
Let us derive an easy consequences of Proposition \ref{prop:FraaDiBruno} that allows us in a sense to reduce the multivariate case to the univariate one.
 
 \begin{lemma} \label{lem:SingleVariable}
 For constants $a_\gamma= a_{|\gamma|}, b_{|\gamma|}  \in 
\mathbb R $ depending only on the 
length of the multiindex $\gamma$ we have
$$
 P^\alpha_{m_1, m_2} (\{a_{|\gamma|}\},\{b_{|\gamma|}\}) = 
P^{|\alpha|}(\{a_{|\gamma|}\}, \{b_{|\gamma|}\}).
$$
 \end{lemma} 
 
 \begin{proof}
   Plugging functions $g$ and $f$ of the form $$g(x_1, \ldots ,x_{m_1} = \tilde 
g( x_1+ \cdots + x_{m_1} ) \cdot (1, \ldots, 1)^t$$ and $$f(y_1, \ldots 
y_{m_2})=\tilde f\left(\frac{y_1+ \cdots + y_{m_2}}{m_2}\right)$$ into the higher 
order chain rule we get from 
$$
 f\circ g = (\tilde f \circ \tilde g) (x_1 + \ldots x_{m_1})
$$
that
$$
P^\alpha_{m_1, m_2} (\{\partial^\gamma f(g(x))\}_{|\gamma| \leq |\alpha|}, \{\partial^\gamma g_i\}_{0 \leq \gamma \leq \alpha}) = P^{|\alpha|} (\{\partial^l \tilde f( \tilde g(x))\}_{l \leq |\alpha|}, \{\partial^l \tilde g\}_{0 \leq l \leq |\alpha|})
$$
So for constants $a_\gamma= a_{|\gamma|}, b_{\gamma}  \in 
\mathbb R $ depending only on the 
length of the multiindex $\gamma$ we have
$$
P^\alpha_{m_1, m_2} (\{a_{|\gamma|}\},\{b_{|\gamma|}\}) = 
P^{|\alpha|}(\{a_{|\gamma|}\}, \{b_{|\gamma|}\}).
$$

 \end{proof}

We will use this lemma to estimate $N_k$.

 \begin{lemma} \label{lem:estimateComposition}
  We have
 \begin{align*}
 N_k
& \leq C P^{k} ( \{ \tilde N_l\}, \{ \tilde M_l\}_{l=0, \ldots, k }).
 \end{align*}
 \end{lemma}
 
 \begin{proof}
Applying Faa di Brunos formula to 
$f\circ g$ where
$$
 g(x) = (x, u(x))
$$
we get
 \begin{align*}
   \partial ^\alpha (f(x,u))\ & = P^{\alpha}_{n,n+1} 
(\{\partial^\gamma f \}, \{\partial^\gamma g \}) 
\end{align*}
where $P^{\alpha}_{n,n +1} 
(\{\partial^\gamma f \}, \{\partial^\gamma g \}) $ is a linear combination with positive coefficients of terms of the form 
$$
  \partial^m_{x_{i_1}, \ldots,  x_{i_m}}f( g(x)) \, \partial^{\gamma_1}g_{i_1} 
\cdots  \partial^{\gamma_k}g_{i_m}  
 $$
 with $1 \leq m \leq |\alpha|$ and $\gamma_1 + \ldots + \gamma_m = \alpha.$ Note 
that due to the special structure of $g$ we have $\partial^\gamma g_i = 1$ for 
$i=1, \ldots, n$ and $|\gamma|=1$ and $\partial^\gamma g_i = 0$ for $i=1, \ldots, n$ 
and $|\gamma|\geq 2$.
Hence,
 \begin{align*}
  (1-r)^{|\alpha|+s} \|\partial^m_{x_{i_1}, x_{i_m}} f( g(x))  
\partial^{\gamma_1}g_{i_1} \cdots  \partial^{\gamma_k}g_{i_m}\|_{(B_r(0))}
  & \leq   \| \partial^k_{x_{i_1}, x_{i_k}} f( g(x)) \|_{B_r(0)}
 \tilde M_{|\gamma_1|} \cdots  \tilde M_{|\gamma_m|} \\
 & \leq  \tilde  N_k
 \tilde M_{|\gamma_1|} \cdots  \tilde M_{|\gamma_m|}.
 \end{align*}
 We hence deduce using Lemma \ref{lem:SingleVariable} that  
  \begin{align*}
   \|  (1-r)^{|\alpha|+s }\partial ^\alpha (f (x,u))\|_{B_r(0)} 
& \leq P^{\alpha} ( \{ \tilde N_k\}, \{ \tilde M_k\}_{k=0, \ldots, |\alpha| }).
 \end{align*}
 Applying this estimate for all multiindices $\alpha \in \mathbb N_0^n$ with $|\alpha|=k$ proves the claim.
 \end{proof}
 
 \subsection{Conclusion of the Proof}
 Combining \eqref{eq:RecursiveInequality2}  with Lemma \ref{lem:estimateComposition} we get
\begin{multline*}
 \tilde M_{k+1} \leq C (P^{k}(\{ \tilde N_l\}, \{ \tilde M_l\}) + k \sum_{l=0}^k \binom {k} {l} \tilde M_{k-l}(2eH)^l l!)
  \\ \leq C (P^{k}(\{ A^l l!\}, \{ \tilde M_l\})+  C k \sum_{l=0}^k \binom {k} {l}  \tilde M_{k-l} A^l l!
\end{multline*}
where again $A:= \sup \{A_f, 2eH\}$.
As above we conclude comparing this with the solution to the initial value problem
$$
\begin{cases}
 c'(t)= G(c(t)) + (tG(t)c(t))', \\
 c(0) = M_0.
 \end{cases} 
$$

\bibliographystyle{plain}
\bibliography{Master}
\end{document}